\documentclass[10pt]{article}

\usepackage{latexsym, amssymb, amsmath, amsthm, amscd}
\usepackage[all,cmtip]{xy}
\usepackage{amsmath}
\usepackage{amsthm}
\usepackage{amssymb}
\usepackage{amsfonts}
\usepackage{amsrefs}
\usepackage{color,authblk}
\usepackage{tikz}

\usepackage{amscd} 
\usepackage{comment}
\usepackage{mathrsfs}
\usepackage{comment}
\usepackage[all]{xy}
\usepackage{graphicx}
\usepackage{authblk}
\usepackage{hyperref}
\usepackage{fullpage}
\usepackage[all,cmtip]{xy}
\usepackage{relsize}
\usepackage{amsmath,amscd}
\usepackage{amssymb}
\usepackage{stackrel}
\usepackage{tikz-cd}
\usepackage{tikz}
\usetikzlibrary{arrows}
\usetikzlibrary{matrix}
\usepackage[mathscr]{euscript}

\usepackage[all]{xy}

\usepackage{hyperref}

\numberwithin{equation}{section} \DeclareMathSizes{2}{10}{12}{13}
\makeatletter
\newcommand*{\doublerightarrow}[2]{\mathrel{
  \settowidth{\@tempdima}{$\scriptstyle#1$}
  \settowidth{\@tempdimb}{$\scriptstyle#2$}
  \ifdim\@tempdimb>\@tempdima \@tempdima=\@tempdimb\fi
  \mathop{\vcenter{
    \offinterlineskip\ialign{\hbox to\dimexpr\@tempdima+1em{##}\cr
    \rightarrowfill\cr\noalign{\kern.5ex}
    \rightarrowfill\cr}}}\limits^{\!#1}_{\!#2}}}

\newcommand{\leftrarrows}{\mathrel{\raise.75ex\hbox{\oalign{%
  $\scriptstyle\leftarrow$\cr
  \vrule width0pt height.5ex$\hfil\scriptstyle\relbar$\cr}}}}
\newcommand{\lrightarrows}{\mathrel{\raise.75ex\hbox{\oalign{%
  $\scriptstyle\relbar$\hfil\cr
  $\scriptstyle\vrule width0pt height.5ex\smash\rightarrow$\cr}}}}
\newcommand{\Rrelbar}{\mathrel{\raise.75ex\hbox{\oalign{%
  $\scriptstyle\relbar$\cr
  \vrule width0pt height.5ex$\scriptstyle\relbar$}}}}

\makeatletter
\def\leftrightarrowsfill@{\arrowfill@\leftrarrows\Rrelbar\lrightarrows}
\newcommand{\xleftrightarrows}[2][]{\ext@arrow 3399\leftrightarrowsfill@{#1}{#2}}
\makeatother

\newtheorem{thm}{Proposition}[section]
\newtheorem{Thm}[thm]{Theorem}

\newtheorem{lem}[thm]{Lemma}
\newtheorem{defn}[thm]{Definition}

\title{Adjunctions between Eilenberg-Moore categories and a PBW-type theorem}

\author{Mamta Balodi\footnote{Stat-Math Unit, Indian Statistical Institute, Bangalore. Email: mamta.balodi@gmail.com} $\qquad\qquad$ Abhishek Banerjee\footnote{Department of Mathematics, Indian Institute of Science, Bangalore. Email: abhishekbanerjee1313@gmail.com} $\qquad\qquad$ Anita Naolekar \footnote{Stat-Math Unit, Indian Statistical Institute, Bangalore. Email: anita.naolekar@gmail.com}}

\date{}

\begin{document}

\maketitle 

\medskip

\begin{abstract}  Recently, Dotsenko and Tamaroff have shown that a morphism of $T\longrightarrow S$ of monads over a category $\mathscr C$ satisfies the PBW-property if and only if it makes $S$ into a free right $T$-module. We consider an adjunction $\Psi=(G,F)$ between categories $\mathscr C$, $\mathscr D$, a monad $S$ on $\mathscr C$ and a monad $T$ on $\mathscr D$. We show that a morphism $\phi:(\mathscr C,S)\longrightarrow (\mathscr D,T)$ that is well behaved with respect to the adjunction $\Psi$ has a PBW-property if and only if it makes $S$ satisfy a certain freeness condition with respect to $T$-modules with values in $\mathscr C$.
\end{abstract}

\smallskip

{\bf MSC(2020) Subject Classification:} 16D90, 18C20  


\smallskip

\section{Introduction}

Let $\mathscr C$ be a category and let $\phi: T\longrightarrow S$ be a morphism of monads on $\mathscr C$. Let $\mathscr C^S$ and $\mathscr C^T$ denote respectively the Eilenberg-Moore categories with respect to $S$ and $T$. Then, if $x\in \mathscr C$ carries the structure of an $S$-algebra, given by a structure map $\lambda_x:Sx\longrightarrow x$, the composition
$Tx\xrightarrow{\phi x}Sx\xrightarrow{\lambda_x}x$ gives $x
$ the structure of a $T$-algebra. This determines a functor $\phi_*:\mathscr C^S\longrightarrow \mathscr C^T$. Moreover, this functor $\phi^*$ has a left adjoint
\begin{equation}
\phi^!:\mathscr C^T\longrightarrow \mathscr C^S
\end{equation} which may be viewed as a ``universal enveloping algebra.'' In \cite{Dot}, Dotsenko and Tamaroff prove a striking result; they say that a morphism of monads on $\mathscr C$ has the PBW property if the underlying object of the universal enveloping algebra of an object in $\mathscr C^T$ does not depend on its $T$-algebra structure. Then, Dotsenko and Tamaroff \cite{Dot} show that, if the Eilenberg-Moore categories contain reflexive coequalizers, $\phi:T\longrightarrow S$ satisfies the PBW-property if and only if $\phi$ makes $S$ into a free right $T$-module.

\smallskip
In our situation, we start with an adjunction $\Psi=(G,F)$ between categories $\mathscr C$ and $\mathscr D$. We consider a monad $S$ on $\mathscr C$ and a monad $T$ on $\mathscr D$. By a $\Psi$-morphism  $\phi=(\phi_G,\phi_F):(\mathscr C,\mathscr S)\longrightarrow (\mathscr D,T)$, we mean a pair of natural transformations $\phi_G:GT\longrightarrow TS$ and
$\phi_F:TF\longrightarrow FS$ satisfying certain condition that make the ``morphism of monads'' compatible with the adjunction $\Psi=(G,F)$. We then define a PBW-property similar to that of \cite{Dot}. Our aim is to show that a $\Psi$-morphism has the PBW-property if and only if it satisfies a certain freeness condition with respect to $T$-modules with values in $\mathscr C$.

\section{Adjunctions between Eilenberg-Moore categories}

We  begin by fixing some notation. Suppose that $(G,F)$ is a pair of adjoint functors between categories $\mathscr A$ and $\mathscr B$. Accordingly, for objects $a\in \mathscr A$ and $b\in \mathscr B$, we have an isomorphism
\begin{equation}
\mathscr B(Ga,b)\cong \mathscr A(a,Fb)
\end{equation} For a morphism $g\in \mathscr B(Fa,b)$, we denote by $g^R$ the corresponding element in $\mathscr A(a,Gb)$. Similarly, for any $f\in \mathscr A(a,Fb)$, we will denote by $f^L$ the corresponding element in $\mathscr B(Ga,b)$. 

\smallskip
A monad on a category $\mathscr C$ consists of an endofunctor $S:\mathscr C\longrightarrow \mathscr C$ along with natural transformations
$\mu(S):S^2\longrightarrow S$ and $\eta(S):1\longrightarrow S$ satisfying the usual associativity and unit conditions. We will  denote this monad as a pair
$(\mathscr C,S)$. A map of such pairs $(\mathscr C,S)\longrightarrow (\mathscr D,T)$ consists of a functor $F:\mathscr C\longrightarrow \mathscr D$ along with a natural transformation $ \phi_F:TF\longrightarrow FS$ such that the following diagrams commute
\begin{equation}\label{2.2cd}
\begin{array}{lll}
\begin{CD}
\small \xymatrix{
T^2F \ar[r]^{T \phi_F}\ar[d]_{\mu(T)F}& TFS\ar[r]^{ \phi_F S} & FS^2 \ar[d]^{F\mu(S)}\\
TF\ar[rr]^{ \phi_F}&& FS\\
}\end{CD} &\qquad \qquad & 
\begin{CD}
\small \xymatrix{
&F\ar[dl]_{\eta(T)F}\ar[dr]^{F\eta(S)}& \\
TF\ar[rr]^{\phi_F}&&FS\\
}
\end{CD}\\
\end{array}
\end{equation} An $S$-algebra, or more precisely, an algebra over the monad $(\mathscr C,S)$ consists of an object $x\in\mathscr C$ along with a morphism $\lambda_x: Sx\longrightarrow x$ which satisfies associativity and unit conditions with respect to the action of $S$, i.e., $\lambda_x\circ \mu(S)x=\lambda_x\circ (S\lambda_x)$ and $1_x=\lambda_x\circ \eta(S)x$. The  category of algebras over the monad $(\mathscr C,S)$ is the Eilenberg-Moore category which we denote by $\mathscr C^S$. A morphism $f:(x,\lambda_x)\longrightarrow (x',\lambda_{x'})$ in $\mathscr C^S$  is given by 
$f:x\longrightarrow x'$ such that $\lambda_{x'}\circ Sf=f\circ \lambda_x$.

\smallskip
Given a monad $\mathscr C^S$, we also denote by $\rho(S):\mathscr C\longrightarrow \mathscr C^S$ the free functor which is left adjoint to the forgetful functor $\pi(S):\mathscr C^S\longrightarrow S$. The following result is well known.

\begin{lem}\label{L2.1}
Let $(F,\phi_F):(\mathscr C,S)\longrightarrow (\mathscr D,T)$ be a morphism of monads. Then, there is a functor $\hat\phi_F:\mathscr C^S\longrightarrow 
\mathscr D^T$ that fits into the following commutative diagram
\begin{equation}\label{2.3cdq}
\begin{CD}
\mathscr C^S @>\hat\phi_F>> \mathscr D^T \\
@V\pi(S)VV @VV\pi(T)V \\
\mathscr C @>F>> \mathscr D\\
\end{CD}
\end{equation}
\end{lem}

\begin{proof}
We consider $(y,\lambda_y)\in \mathscr C^S$. It may be verified that the composition
\begin{equation}
\begin{CD}\lambda_{Fy}:TFy @>\phi_F(y)>> FSy @>F\lambda_y>> Fy 
\end{CD} \end{equation}
gives $Fx$ the structure of an algebra over $(\mathscr D,T)$. This defines a functor $\hat\phi_F:\mathscr C^S\longrightarrow 
\mathscr D^T$ that makes the diagram \eqref{2.3cdq} commute.

\end{proof}

We now fix a pair $\Psi=(G:\mathscr D\longrightarrow \mathscr C,F:\mathscr C\longrightarrow \mathscr D)$ of adjoint functors between $\mathscr D$ and $\mathscr C$. Our aim is to construct a commutative diagram of left adjoints corresponding to the commutative square of right adjoints in \eqref{2.3cdq}.

\begin{defn}\label{D2.2} Let $(\mathscr C,S)$ and $(\mathscr D,T)$ be monads and $\Psi=(G,F)$ be an adjunction between $\mathscr D$ and $\mathscr C$. A $\Psi$-morphism  of monads is a pair $\phi=(\phi_G,\phi_F)$ such that:

\smallskip
(1) $\phi_G: GT\longrightarrow SG$ is a natural transformation which satisfies
\begin{equation}\label{2.1dcr}
\begin{array}{lll}
\begin{CD}
\small \xymatrix{
GT^2 \ar[r]^{\phi_GT}\ar[d]_{G\mu(T)}& SGT\ar[r]^{ S\phi_G } & S^2G \ar[d]^{\mu(S)G}\\
GT\ar[rr]^{ \phi_G}&& SG\\
}\end{CD} &\qquad \qquad & 
\begin{CD}
\small \xymatrix{ G \ar[rr]^{G\eta(T)} \ar[drr]_{\eta(S)G}& & GT \ar[d]^{\phi_G}\\
&& SG\\
}\end{CD}\\
\end{array}
\end{equation}

\smallskip
(2) $\phi_F:TF\longrightarrow FS$ is a natural transformation which satisfies
\begin{equation}\label{2.2dcr}
\begin{array}{lll}
\begin{CD}
\small \xymatrix{
T^2F \ar[r]^{T \phi_F}\ar[d]_{\mu(T)F}& TFS\ar[r]^{ \phi_F S} & FS^2 \ar[d]^{F\mu(S)}\\
TF\ar[rr]^{ \phi_F}&& FS\\
}\end{CD} &\qquad \qquad & 
\begin{CD}
\small \xymatrix{
&F\ar[dl]_{\eta(T)F}\ar[dr]^{F\eta(S)}& \\
TF\ar[rr]^{\phi_F}&&FS\\
}
\end{CD}\\
\end{array}
\end{equation} 

(3) For any morphism $f: Gx\longrightarrow y$ in $\mathscr C$, the following two equivalent conditions are satisfied 
\begin{equation}\label{2.3dcr}
\begin{array}{lll}
\begin{CD}
GTx @>GT(f^R)>> GTFy \\
@V\phi_G(x)VV @VV\phi_F(y)^LV \\
SGx @>Sf>> Sy \\
\end{CD} &\qquad \Leftrightarrow \qquad & \begin{CD}
Tx @>T(f^R)>> TFy \\
@V\phi_G(x)^RVV @VV\phi_F(y)V \\
FSGx @>FSf>> FSy \\
\end{CD} \\
\end{array}
\end{equation}

\end{defn}

\begin{lem}
\label{L2.3} Let $\phi=(\phi_G,\phi_F):(\mathscr C,S)\longrightarrow (\mathscr D,T)$ be a $\Psi$-morphism of monads. Let $(x,\lambda_x)\in \mathscr D^T$. Then, the following two compositions are identity
\begin{equation}\label{2.8eq}
SGx \xrightarrow{\qquad SG\eta(T)x\qquad }SGTx\doublerightarrow{(\mu(S)Gx)\circ (S\phi_G(x))}{\qquad SG\lambda_x\qquad} SGx
\end{equation}
\end{lem}

\begin{proof}
From the condition \eqref{2.1dcr}, we observe that
\begin{equation}\label{2.9fo}
\begin{array}{lll}
\begin{CD}
\small \xymatrix{ Gx \ar[rr]^{G\eta(T)(x)} \ar[drr]_{\eta(S)G(x)}& & GTx \ar[d]^{\phi_G(x)}\\
&& SGx\\
}
\end{CD}
& \qquad\Rightarrow \qquad & 
\begin{CD}
\small \xymatrix{ SGx \ar[rr]^{SG\eta(T)(x)} \ar[drr]_{S\eta(S)G(x)}& & SGTx \ar[d]^{S\phi_G(x)}\\
&& SSGx\\
}
\end{CD} \\
\end{array}
\end{equation} Since $(\mu(S)Gx)\circ (S\eta(S)Gx)=id$, it follows from \eqref{2.9fo} that $(\mu(S)Gx)\circ (S\phi_G(x))\circ ( SG\eta(T)x)=id$. Additionally, since $(x,\lambda_x)$ is an algebra over $(\mathscr D,T)$, we have $(SG\lambda_x)\circ (SG\eta(T)x)=SG(\lambda_x\circ \eta(T)x)=id$.
\end{proof}

From now onwards, we assume that the Eilenberg-Moore categories $\mathscr C^S$ and $\mathscr D^T$ both contain reflexive coequalizers. For more on this condition, we refer the reader to \cite{AK}, \cite[$\S$ 9.3]{BW} and \cite{Lint}. 

\begin{lem}\label{L2.4}
Let $\phi=(\phi_G,\phi_F):(\mathscr C,S)\longrightarrow (\mathscr D,T)$ be a $\Psi$-morphism of monads. Let $(x,\lambda_x)\in \mathscr D^T$. Then, the coequalizer 
\begin{equation}\label{coe210}
Coeq\left(SGTx\doublerightarrow{(\mu(S)Gx)\circ (S\phi_G(x))}{\qquad SG\lambda_x\qquad} SGx\right)
\end{equation} in $\mathscr C$ is equipped with the structure of an $S$-algebra. This determines a functor $\hat\phi_G: \mathscr D^T
\longrightarrow \mathscr C^S$.
\end{lem}
\begin{proof}
We observe that each of the objects in \eqref{2.8eq} is canonically equipped with the structure of an $S$-algebra and each arrow in
\eqref{2.8eq} is a morphism of $S$-algebras. Applying Lemma \ref{L2.3}, it follows that the coequalizer in \eqref{coe210} is a reflexive coequalizer. By assumption, $\mathscr C^S$ contains reflexive coequalizers and the result follows.
\end{proof}

\begin{thm}\label{P2.5} Let $\Psi=(G:\mathscr D\longrightarrow \mathscr C,F:\mathscr C\longrightarrow \mathscr D)$ be a pair of adjoint functors between $\mathscr D$ and $\mathscr C$. Let $\phi=(\phi_G,\phi_F):(\mathscr C,S)\longrightarrow (\mathscr D,T)$ be a $\Psi$-morphism of monads.  Then, $\hat\phi_G: \mathscr D^T
\longrightarrow \mathscr C^S$ is left adjoint to the functor  $\hat\phi_F$ and fits into the following commutative diagram
\begin{equation}\label{com211}
\begin{CD}
\mathscr C^S @<\hat\phi_G<< \mathscr D^T \\
@A\rho(S)AA @AA\rho(T)A\\
\mathscr C @<G<< \mathscr D \\
\end{CD}
\end{equation} 

\end{thm}

\begin{proof}
We consider $(x,\lambda_x)\in \mathscr D^T$, $(y,\lambda_y)\in \mathscr C^S$ and a morphism $\beta:x\longrightarrow Fy$ which is a map of $T$-algebras. 
In particular, $\beta\in \mathscr D(x,Fy)$.  Accordingly, we have a morphism $\beta^L: Gx\longrightarrow y$ in $\mathscr C$. From the coequalizer in 
\eqref{coe210} we see that in order to construct a morphism $\alpha:\hat\phi_G(x)\longrightarrow y$ we must show  that
\begin{equation}\label{212ha}
\lambda_y\circ S\beta^L\circ (\mu(S)Gx)\circ (S\phi_G(x))=\lambda_y\circ S\beta^L\circ (SG\lambda_x)
\end{equation} Since $\beta: x\longrightarrow Fy$ is a morphism of $T$-algebras, the following diagrams commute
\begin{equation}\label{213dw}
\begin{CD}\small \xymatrix{
Tx \ar[r]^{T\beta} \ar[d]^{\lambda_x} & TFy\ar[r]^{\phi_F(y)} &FSy\ar[dl]^{F\lambda_y} \\
x \ar[r]^\beta &Fy & \\ 
}\end{CD} \quad \Rightarrow \quad \begin{CD}\small \xymatrix{
GTx \ar[rr]^{GT\beta} \ar[d]^{G\lambda_x} & &GTFy \ar[rrr]^{\phi_F(y)^L} &&& Sy\ar[dlll]^{\lambda_y} \\
Gx \ar[rr]^{\beta^L} & & y &&& \\ 
}\end{CD}
\end{equation} Here the right hand side diagram follows from the left hand side diagram using the adjointness of $(G,F)$. Applying the functor $S$ to the right hand side diagram in \eqref{213dw} and composing with $\lambda_y$, we obtain
\begin{equation}\label{214dqt}
\small \xymatrix{
SGTx \ar[rr]^{SGT\beta} \ar[d]^{SG\lambda_x} & & SGTFy\ar[rrrr]^{S(\phi_F(y)^L)} &&&&SSy\ar[d]_{S\lambda_y} \ar[r]^{\mu(S)y} &Sy\ar[d]^{\lambda_y}\\
SGx\ar[rrrrrr]^{S\beta^L} & & &&&&Sy \ar[r]^{\lambda_y} & y \\ 
}
\end{equation} Using the compatibility condition in \eqref{2.3dcr} in Definition \ref{D2.2}, we also have the following commutative diagram
\begin{equation}\label{2.8cdsx}
\begin{CD}
\small \xymatrix{
SGTx\ar[rrr]^{SGT\beta} \ar[d]_{S\phi_G(x)}& & & SGTFy\ar[d]^{S(\phi_F(y)^L)} & \\
SSGx \ar[rrr]^{SS\beta^L} \ar[d]_{\mu(S)Gx}&&& SSy\ar[d]^{\mu(S)y} &  \\
SGx\ar[rrr]^{S\beta^L} &&& Sy  \ar[r]^{\lambda_y} & y\\
}\end{CD}
\end{equation} Combining \eqref{214dqt} and \eqref{2.8cdsx}, the equality in \eqref{212ha} follows.  It may be verified directly that $\alpha:\hat\phi_G(x)\longrightarrow y$ is a morphism of $S$-algebras. Conversely, if we have an $S$-algebra morphism $\hat\phi_G(x)\longrightarrow y$, we have an induced map 
$Gx\xrightarrow{\eta(S)Gx}SGx\longrightarrow y$, which corresponds to a morphism $x\longrightarrow Fy$. Again, it may be verified directly that this is a $T$-algebra morphism and these two associations are inverses of each other.

\smallskip
Accordingly, we now have  $ \hat\phi_G: \mathscr D^T
\longrightarrow \mathscr C^S$ which is left adjoint to $\hat\phi_F:\mathscr C^S\longrightarrow \mathscr D^T$.  Then, $ \hat\phi_G\circ \rho(T)$ is left adjoint to $\pi(T)\circ \hat\phi_F$. By Lemma \ref{L2.1}, we have $\pi(T)\circ \hat\phi_F=F\circ \pi(S)$ and hence the diagram \eqref{com211} of left adjoints must commute.
\end{proof}

\section{The Main theorem}

We begin with the following definition.

\begin{defn}\label{D3.f5}
Let $\mathscr C$, $\mathscr D$ be categories and let $(\mathscr D,T)$ be a monad. By a $T$-module taking values in $\mathscr C$, we will mean a pair $(P,\nu(P))$ such that

\smallskip
(a) $P:\mathscr D\longrightarrow \mathscr C$ is a functor

\smallskip
(b) $\nu(P):PT\longrightarrow P$ is a natural transformation satisfying
\begin{equation}\label{3cd2v}
\begin{array}{lll}
\begin{CD}
PT^2 @>P\mu(T)>>PT \\
@V\nu(P)TVV @VV\nu(P)V \\
PT @>\nu(P)>> P\\
\end{CD} & \qquad \qquad
& \begin{CD}
\small \xymatrix{
P \ar[rr]^{P\eta(T)} \ar[drr]_{id} & & PT\ar[d]^{\nu(P)} \\
&& P \\
}
\end{CD}\\
\end{array}
\end{equation} A morphism $\alpha:(P,\nu(P))\longrightarrow (P',\nu(P'))$ of such $T$-modules  consists of a natural transformation $\alpha:P\longrightarrow P'$ of functors satisfying $\nu(P')\circ (\alpha T)=\alpha\circ \nu(P)$. We will denote by $Mod_T^{\mathscr C}$ the category of $T$-modules taking values in $\mathscr C$.
\end{defn}

\begin{thm}\label{P3.2f5}
Let $\mathscr C$, $\mathscr D$ be categories and let $(\mathscr D,T)$ be a monad. There is a canonical functor 
\begin{equation}\label{can3z}\Sigma_T^{\mathscr C}:Fun(\mathscr D,\mathscr C)\longrightarrow Mod_T^{\mathscr C} \qquad P\mapsto (PT,P\mu(T))
\end{equation} which is left adjoint to the forgetful functor $Mod_T^{\mathscr C} \longrightarrow Fun(\mathscr D,\mathscr C)$. 
\end{thm}

\begin{proof}
Since $(\mathscr D,T)$ is a monad, it is easy to verify that for any $P\in Fun(\mathscr D,\mathscr C)$, the pair $ (PT,P\mu(T))$ satisfies the conditions in \eqref{3cd2v} for being a $T$-module with values in $\mathscr C$. It remains to show that for any $(P',\nu(P'))\in  Mod_T^{\mathscr C} $, we have isomorphisms
\begin{equation}
\label{33eq}
 Mod_T^{\mathscr C}( (PT,P\mu(T)),(P',\nu(P')))\cong Fun(\mathscr D,\mathscr C)(P,P')
\end{equation} Indeed, given $\alpha:P\longrightarrow P'$ in $Fun(\mathscr D,\mathscr C)$, we have $\alpha^L: (PT,P\mu(T))\longrightarrow (P',\nu(P'))$ defined by setting
\begin{equation}
\begin{CD}
\alpha^L:PT @>\alpha T>> P'T @>\nu(P')>> P'\\
\end{CD} 
\end{equation} Conversely, for $\beta: (PT,P\mu(T))\longrightarrow (P',\nu(P'))$ in $Mod_T^{\mathscr C}$, we have the transformation $\beta^R:P\longrightarrow P'$ given by
\begin{equation}
\begin{CD}
\beta^R: P @>P\eta(T)>> PT @>\beta>> P'
\end{CD}
\end{equation} Because $(\mathscr D,T)$ is a monad, it is clear that these two associations are inverse to each other.
\end{proof}

Following Proposition \ref{P3.2f5}, we will say that $ (PT,P\mu(T))$ is the free $T$-module in $\mathscr C$ associated to the functor $P:\mathscr D\longrightarrow \mathscr C$. 

\begin{lem}\label{L3.3cj}
Let $\Psi=(G,F)$ be an adjunction and let   $\phi=(\phi_G,\phi_F):(\mathscr C,S)\longrightarrow (\mathscr D,T)$ be a $\Psi$-morphism of monads. Then, $SG:\mathscr D\longrightarrow \mathscr C$ is canonically equipped with the structure of a $T$-module taking values in $\mathscr C$.
\end{lem}

\begin{proof}
We set $\nu': SGT\xrightarrow{S\phi_G}  S^2G\xrightarrow{\mu(S)G} SG $. In order to show that this determines a $T$-module, we note  that the following diagram commutes. 
\begin{equation}\label{36cdu}\small\begin{CD}
\small \xymatrix{
SGT^2 \ar[dd]_{SG\mu(T)} \ar[rr]^{S\phi_GT} & & SSGT \ar[d]_{SS\phi_G} \ar[rr]^{\mu(S)GT}& & SGT \ar[d]^{S\phi_G}\\
& & SSSG \ar[rr]^{\mu(S)SG}\ar[d]_{S\mu(S)G}&& SSG\ar[d]_{\mu(S)G} \\
SGT \ar[rr]_{S\phi_G}&&S^2G \ar[rr]_{\mu(S)G}&& SG \\
}\end{CD}
\end{equation} In \eqref{36cdu}, the left hand square is obtained by applying $S$ to the commutative square in \eqref{2.1dcr}. Using the triangle in \eqref{2.1dcr}, we also have the commutative diagram
\begin{equation}\label{37deq}\small
\begin{CD}
\small \xymatrix{  & & SG\ar[dll]_{SG\eta(T)} \ar[d]_{S\eta(S)G} \ar[drr]^1 & \\
SGT\ar[rr]^{S\phi_G}&& SSG \ar[rr]^{\mu(S)G} && SG\\
}
\end{CD}
\end{equation} This proves the result.
\end{proof}

\begin{defn}\label{D3.4}
Let $\Psi=(G,F)$ be an adjunction and let   $\phi=(\phi_G,\phi_F):(\mathscr C,S)\longrightarrow (\mathscr D,T)$ be a $\Psi$-morphism of monads. We will say that $\phi$ satisfies the PBW-property if there exists a functor $Q:\mathscr D\longrightarrow \mathscr C$ which fits into the following commutative diagram
\begin{equation}\label{pbwdiag}
\begin{CD}
\small \xymatrix{
\mathscr D^T  \ar[r]^{\hat\phi_G}\ar[d]_{\pi(T)}& \mathscr C^S  \ar[d]_{\pi(S)} \\
\mathscr D\ar[r]^{ Q}& \mathscr C\\
}\end{CD} 
\end{equation}
\end{defn}

\begin{thm}\label{P3.6b} 
Let $\Psi=(G,F)$ be an adjunction and let   $\phi=(\phi_G,\phi_F):(\mathscr C,S)\longrightarrow (\mathscr D,T)$ be a $\Psi$-morphism of monads. Then, if $\phi=(\phi_G,\phi_F)$ has the PBW property, $\nu': SGT\xrightarrow{S\phi_G}  S^2G\xrightarrow{\mu(S)G} SG $ gives $SG$ the structure of a free $T$-module in $\mathscr C$.
\end{thm}

\begin{proof}
Since $\phi$ has the PBW property, we have $Q:\mathscr D\longrightarrow \mathscr C$ which fits into the commutative diagram \eqref{pbwdiag}. We first claim that $SG=QT$. Using \eqref{com211} and \eqref{pbwdiag}, we have the following commutative diagram.
\begin{equation}\label{39com}
\begin{CD}
\mathscr D @>\rho(T)>> \mathscr D^T @>\pi(T)>> \mathscr D \\
@VGVV @V\hat\phi_GVV @VVQV \\
\mathscr C @>\rho(S)>> \mathscr C^S @>\pi(S)>> \mathscr C \\
\end{CD}
\end{equation} From \eqref{39com}, it is clear that for any $x\in \mathscr D$, we have $QTx=SGx\in \mathscr C$. It remains to show that $QT=SG$ as objects of $Mod_T^{\mathscr C}$.  For this, we note the following commutative diagram
\begin{equation}\label{310vq}\small
\begin{CD}
\small \xymatrix@C+2pc{
&SGTy\ar[rrd]^{S\phi_Gy}&&\\
SSGTy \ar[ur]^{\mu(S)GTy}\ar[rr]_{SS\phi_Gy}&&SSSGy\ar[d]_{S\mu(S)Gy}\ar[r]_{\mu(S)SGy}&SSGy\ar[d]^{\mu(S)Gy}\\
SGTTy \ar[r]_{SG\mu(T)y} \ar[u]^{S\phi_GTy}& SGTy\ar[r]_{S\phi_G(Ty)} & SSGy\ar[r]_{\mu(S)Gy} & SGy \\
}
\end{CD}
\end{equation}
which gives for any $y\in \mathscr D$ the morphism from the coequalizer $QTy=Q\pi(T)(Ty,\lambda_{Ty})=\pi(S)\hat\phi_G(Ty,\lambda_{Ty})\longrightarrow SGy$. The condition \eqref{2.1dcr} gives us the commutative diagram 
\begin{equation}\label{311fn}\small
\begin{CD}
\small \xymatrix@C+2pc{
SGTTTy\ar[r]_{SG\mu(T)Ty}\ar[dd]_{SGT\mu(T)y}&SGTTy\ar[r]_{S\phi_GTy}\ar[dd]_{SG\mu(T)y}&SSGTy\ar[d]_{SS\phi_Gy}\ar[r]_{\mu(S)GTy}&SGTy\ar[d]^{S\phi_Gy}\\
&&SSSGy\ar[d]_{S\mu(S)Gy}\ar[r]_{\mu(S)SGy}&SSGy\ar[d]^{\mu(S)Gy}\\
SGTTy \ar[r]_{SG\mu(T)y} & SGTy \ar[r]_{S\phi_Gy} & SSGy \ar[r]_{\mu(S)Gy} & SGy\\
}
\end{CD}
\end{equation} and also
\begin{equation}\label{312fn}\small
\begin{CD}
\small \xymatrix@C+2pc{
SGTTTy\ar[r]_{S\phi_GTTy}\ar[dd]_{SGT\mu(T)y}&SSGTTy\ar[r]_{\mu(S)GTTy}\ar[dd]_{SSG\mu(T)y}&SGTTy\ar[r]_{S\phi_GTy}\ar[dd]_{SG\mu(T)y}&SSGTy\ar[d]_{SS\phi_Gy}\ar[r]_{\mu(S)GTy}&SGTy\ar[d]^{S\phi_Gy}\\
&&&SSSGy\ar[d]_{S\mu(S)Gy}\ar[r]_{\mu(S)SGy}&SSGy\ar[d]^{\mu(S)Gy}\\
SGTTy\ar[r]_{S\phi_GTy}&SSGTy \ar[r]_{\mu(S)GTy} & SGTy \ar[r]_{S\phi_Gy} & SSGy \ar[r]_{\mu(S)Gy} & SGy\\
}
\end{CD}
\end{equation} Putting together \eqref{311fn} and \eqref{312fn} and considering the coequalizers defining $\hat\phi_G(Ty,\lambda_{Ty})$ and $\hat\phi_G(TTy,
\lambda_{TTy})$, we see that  $QT=SG$ is compatible with the right action of $T$ on both sides. This proves the result.
\end{proof}

\begin{thm}\label{P3.7b} 
Let $\Psi=(G,F)$ be an adjunction and let   $\phi=(\phi_G,\phi_F):(\mathscr C,S)\longrightarrow (\mathscr D,T)$ be a $\Psi$-morphism of monads. Then,  $SG\in Mod_T^{\mathscr C}$ is isomorphic  to a free $T$-module taking values in $\mathscr C$. 
\end{thm}

\begin{proof}
We take $Q:\mathscr D\longrightarrow \mathscr C$ so that $SG\cong QT$ in $Mod_T^{\mathscr C}$. We claim that $Q\pi(T)=\pi(S)\hat\phi_G$. Using the definition of $\hat\phi_G$ in \eqref{coe210}, this means that we must show that the coequalizer
\begin{equation}\label{coe313}
Coeq\left(SGTx\doublerightarrow{(\mu(S)Gx)\circ (S\phi_G(x))}{\qquad SG\lambda_x\qquad} SGx\right)\cong Q(x)
\end{equation}
in $\mathscr C$ for any $T$-algebra $(x,\lambda_x)\in \mathscr D^T$. Since $SG\cong QT$ as $T$-modules, it follows that the composition
$SGT\xrightarrow{S\phi_G}SSG\xrightarrow{\mu(S)G}SG$ giving the $T$-module structure on $SG$ corresponds to the composition
$QTT\xrightarrow{Q\mu(T)}QT$ giving the $T$-module structure on $QT$. Accordingly, we can prove \eqref{coe313} by showing 
that
\begin{equation}\label{coe314}
Coeq\left(QTTx\doublerightarrow{Q\mu(T)x}{\qquad QT\lambda_x\qquad}QTx\right)\cong Q(x)
\end{equation} We will prove this by showing that the following is a split coequalizer diagram in $\mathscr C$:
\begin{equation}\label{sfork315}
QTx\xrightarrow{\epsilon=Q\eta(T)Tx}QTTx \doublerightarrow{\gamma=Q\mu(T)x}{\qquad \delta=QT\lambda_x\qquad}QTx
\xleftrightarrows[\text{$\qquad \alpha=Q\lambda_x\qquad$}]{\text{$\qquad \beta=Q\eta(T)x\qquad$}}  Qx
\end{equation} Since $(x,\lambda_x)$ is a $T$-algebra, we see that $\alpha\gamma=\alpha\delta$. Since $T$ is a monad, we get $\alpha\beta=id$ and $\gamma\epsilon=id$. We also observe directly that $\delta\epsilon = \beta\alpha$. We conclude that $Qx=Coeq(\gamma,\delta)$, which proves the result.
\end{proof}

\begin{Thm}\label{PB3W}
Let $\Psi=(G,F)$ be an adjunction and let   $\phi=(\phi_G,\phi_F):(\mathscr C,S)\longrightarrow (\mathscr D,T)$ be a $\Psi$-morphism of monads. Then, the following are equivalent.

\smallskip
(1) There exists a functor $Q:\mathscr D\longrightarrow \mathscr C$ such that there are isomorphisms in $\mathscr C$
\begin{equation}\label{316wbp}
\hat\phi_G(x,\lambda_x)=Q(x) \qquad \forall\textrm{ }(x,\lambda_x)\in \mathscr D^T
\end{equation} natural with respect to morphisms in $\mathscr D^T$.

\smallskip
(2)  The functor $SG$, along with the map $\nu': SGT\xrightarrow{S\phi_G}  S^2G\xrightarrow{\mu(S)G} SG $, is isomorphic to a free $T$-module with values in $\mathscr C$.
\end{Thm}

\begin{proof}
(1) $\Rightarrow$ (2) follows from Proposition \ref{P3.6b} and (2) $\Rightarrow$ (1) follows from Proposition \ref{P3.7b}.
\end{proof}

\end{document}